\documentclass[a4,12pt, leqno]{amsart}
\usepackage{amsmath,amsfonts,amsthm,latexsym,amssymb,mathrsfs}
\usepackage{bbm}

\def\R{\mathcal{R}}

\newtheorem{df}{Definition}[section]
\newtheorem{thm}[df]{Theorem}
\newtheorem{pro}[df]{Proposition}
\newtheorem{cor}[df]{Corollary}

\newtheorem{rema}[df] {Remark}

\begin{document}
\setcounter{page}{1}

\title[The inverse along an element]{The inverse along an element in rings}

\author[J. Ben\'{\i}tez and E. Boasso]{Julio Ben\'{\i}tez and Enrico Boasso}

\begin{abstract}In this article several properties of the inverse along an element will be studied in the context of unitary
rings. New characterizations of the existence of this inverse will be proved. Moreover, the set of
all invertible elements along a fixed element will be fully described. Futhermore, commuting inverse
along an element will be characterized. The special cases of the group inverse, the (generalized) Drazin inverse
and the Moore-Penrose inverse (in rings with involutions) will be also considered.  \par
\end{abstract}
\maketitle

\section{Introduction}

\noindent In \cite{M} a notion of an inverse along an element was introduced. This inverse has the advantage
that it encompasses several well known generalized inveses such as the group inverse, the Drazin inverse
and the Moore-Penrose inverse. The aforementioned inverse was studied by several authors (see for example \cite{CKM, KM, M, M2, MP, ZCP}).\par

\indent The main objective of this article is to study several properties of the inverse along an element in unitary rings.   In section 3,
after having recalled some preliminary definitions and facts in section 2, more equivalent conditions that assure the existence
of the inverse under consideration will be given. In addition, in section 4 more characterizations of this inverse will be proved.
In section 5, the set of all invertible elements along a fixed element will be fully described.
Furthermore, the reverse order law and the special cases of 
the group inverse, the (generalized) Drazin inverse and the Moore-Penrose inverse (in the presence of an involution) will be 
considered. In section 6 commuting inverse along an element will be characterized. In particular, a characterization of group invertible elements
through the inverse along an element will be presented. Finally, in section 7 inverses along  elements that are also 
inner inverses will be considered.

\section{Preliminary definitions and facts}

\noindent From now on $\R$ will denote a unitary ring with unit 1. Let $\R^{-1}$ be the set of invertible elements of $\R $ and denote by
 $\R^\bullet$ the set of idempotents of $\R$, i.e., 
$\R^\bullet=\{p\in\R\colon p^2=p\}$. Given $a\in\R$, the following notation will be used:
$a\R=\{ax\colon x\in\R\}$, $\R a=\{xa\colon x\in\R\}$, $a^{-1} (0)=\{x\in\R\colon ax=0\}$,  
$a_{-1} (0)=\{x\in\R\colon xa=0\}$.\par

\indent Recall that $a\in\R$ is said to be \it regular\rm, if there is $z\in\R$ such that $a=aza$.
In addition, the element $z$ will be said to be an \it inner \rm or a  \it generalized inverse \rm of $a$.
The set of regular elements of $\R$ will be denoted by $\hat{\R}$. 
Next follows one of the main 
definitions of this article.\par

\begin{df} Given a unitary ring $\R$ and $a\in\R$, an element $b\in\R$ is said to be an outer inverse of $a$,
if $b=bab$.\end{df}

\indent The following facts will be used in the article.\par

\begin{rema}\label{rema1}\rm Let $\R$ be a unitary ring and consider $a$, $b\in\R$ such that $b$
is an outer inverse of $a$. Then, the following statements can be easily proved:\par
\noindent (i) $ab$, $ba\in\R^\bullet$.\par
\noindent (ii) $bR=ba\R$, $\R ab=\R b$.\par
\noindent (iii) $b^{-1}(0)=(ab)^{-1}(0)$, $b_{-1}(0)=(ba)_{-1}(0)$.\par
\noindent (iv) When $b$ is also an inner inverse of $a$, i.e., $a=aba$, it not difficult to prove that  the following statements are equivalent:\par
\indent (1) $a=aba$,\par
\indent (2) $\R=b\R \oplus a^{-1} (0)$,\par
\indent (3) $\R=\R b \oplus a_{-1}(0)$.
\end{rema}

\indent Next the definition of the key notion of this article will be recalled (see \cite[Definition 4]{M}).\par
 
\begin{df} \label{def2}Let $\R$ be a unitary ring and consider $a$, $d\in\R$. The element $a$ will be said to be invertible along $d$, if there exists $b\in\R$ such that
 $b$ is an outer inverse of $a$ and $b\R=d\R$ and $\R b=\R d$.\end{df}

\indent Recall that, in the conditions of Definition \ref{def2}, according to \cite[Theorem 6]{M}, if such $b\in\R$ exists, then it is unique.
Therefore, the element $b$ satisfying  Definition \ref{def2} will be said to be the \it inverse of $a$ along $d$. \rm
In this case, the inverse under consideration will be denoted by $a^{\parallel d}$. Note that if $\tilde{d}\in \R$ is such that
$d\R= \tilde{d}\R$ and $\R d=\R \tilde{d}$, then $a$ is invertible along $d$ if and only if $a$ is invertible along $\tilde{d}$,
in addition, in this case $a^{\parallel d}=a^{\parallel \tilde{d}}$. In particular, given $d\in\R^{-1}$, necessary and sufficient for $a\in\R$ to be invertible along $d$ is that $a\in \R^{-1}$; moreover, in this case $a^{\parallel d}=a^{-1}$. In fact, since $1\R=d\R$ and $\R 1=\R d$,
$a$ is invertible alonf $d$ if and only if $a$ is invertible along $1$, which is equivalent to $a\in\R^{-1}$, and in this case
$a^{\parallel 1}=a^{-1}$. 

\indent Moreover, according to \cite[p. 3]{MP}, if $a^{\parallel d}$ exists, then $d$ is regular. 
In particular, if $d$ is not regular, no $a\in\R$ has an inverse along $d$. Thus, without loss of generality it will be assumed that
$d\in\hat{\R}$. However, more restriction on the element $d$ can be stated, as the next remak will show.\par

\begin{rema}\label{rema7}\rm Let $\R$ be a unitary ring and consider $d\in\hat{\R}$. Note that if $d=0$, then any $a\in\R$ 
is invertible along $0$; in fact, in this case $a^{\parallel d}=0$. Then suppose that $d\neq 0$ and that there exists $\overline{d}\in\R$ such that
$d\overline{d}d=d$. Thus, $d(1-\overline{d}d)=0=(1-d\overline{d})d$. Therefore, if $d\in\R$ is not a zero divisor (there is no $z\in \R$, $z\neq 0$, such that
$zd=0$ or $dz=0$), then $d\in \R^{-1}$. Consequently, in the general case, it is possible to assume that $d\in\hat{\R}\setminus (\R^{-1}\cup \{0\})$, with $d$
a zero divisor.
However, if the ring $\R$ has no zero divisors, for example if $\R$ is a field, an integral domain or a polynomial ring over an integral domain,
an inverse along an element $d\in\R$ exists if and only if $d\in \R^{-1}\cup \{0\}$, in which case this situation has been characterized.
\end{rema}

\indent Next follow the definitions of several generalized inverses such as the group inverse, the (generalized) Drazin inverse and 
the Moore-Penrose inverse. These classes of invertible elements are particular cases of  the inverse studied in this article.\par

\indent Let $\R$ be a unitary ring and consider $a\in \R$. The element $a\in\R$ will be said to be \it group invertible, \rm if there exists a (necessarily unique)
$b\in \R$ such that 
$$a=aba,\hskip.3truecm b=bab,\hskip.3truecm  ab=ba
$$ 
\noindent (see for example \cite{RS}). When $a\in \R$ is group invertible, its group inverse will be denoted by $a^\sharp$. Clearly, $a^\sharp$ is group invertible and $(a^\sharp)^\sharp=a$. According to  \cite[Theorem 11]{M}, necessary and sufficient for $a\in \R$ to be group invertible is that $a$ is invertible along $a$,
what is more, in this case $a^\sharp=a^{\parallel a}$. Next some of the main properties of group invertible elements will be recalled. To this end, let $\R^\sharp$  stand for the set of all group invertible elements of the ring $\R$. In addition, recall that if $\R$ is  unitary ring and $p\in \R^\bullet$, then $p\R p$ is a subring of $\R$ with unit $p$. 

\begin{rema}\label{rema13}\rm Let $\R$ be a unitary ring and consider $a\in\R^\sharp$. \par
\noindent (i). Note that 
\begin{align*}
a\R&=aa^\sharp \R=a^\sharp a \R= a^\sharp \R,\\
\R a&=\R aa^\sharp =\R a^\sharp a =\R a^\sharp,\\
aa^\sharp \R aa^\sharp&=a^\sharp a \R a^\sharp a=a \R a=a^\sharp \R a^\sharp.\\
\end{align*}
\noindent (ii). Recall that according to \cite[Lemma 3]{RS}, $a\in\R^{\sharp}$ if and only if
there exists $p\in\R^\bullet$ such that $a+p\in\R^{-1}$, $ap=pa=0$. Now, using this result, it is not difficult to prove that
necessary and sufficient for $a$ to be group invertible is that there exists $p\in  \R^\bullet$
such that $a=(1-p)a(1-p)$ and $a\in ((1-p)\R (1-p))^{-1}$. In this case, if $b\in ((1-p)\R (1-p))^{-1}$
is such that $ab=ba=1-p$, then $a^\sharp=b$. Moreover, 
according to \cite[Corollary 2]{RS}, the idempotent involved in the definition of a group invertible element is unique, and if it is denoted by 
$p_a$, then  $p_a=1-a^\sharp a$. In particular, $p_a=p_{a^\sharp}$.\par
\noindent (iii). Let $a\in \R$ and suppose that $a$ has a commuting generalized inverse $b$,
i.e., $aba=a$ and $ab=ba$. Then, it is not difficult to prove that $a\in\R^\sharp$ and $a^\sharp=bab$. \par
\noindent (iv). Let $a\in\R^\sharp$ and consider $n\in\mathbb{N}$. Then, an easy calculation proves that
$a^n\in\R^\sharp$, $(a^n)^\sharp=(a^\sharp)^n$ and $p_{a^n}=p_a$.
\end{rema}

\indent Let $\R$ be a unitary ring and consider $a\in\R$. The element $a$ is said to be \it Drazin invertible, \rm
if there exists a (necessarily unique) $x\in \R$ such that 
$$
a^mxa=a^m,\, xax=x,\, ax=xa,
$$
for some $m\in\mathbb{N}$ (see for example \cite{D, RS}). In this case, the solution of these equations will be denoted by $a^d$
and $\R^D$ will stand for the set of all Drazin invertible elements of $\R$. In addition,
the smallest $m$ for which the above equations hold is called the \it Drazin index  \rm of $a$
and it will be denoted by ind$(a)$. Note that ind$(a)=1$ if and only if $a$ is group invertible. On the
other hand, it is not difficult to prove that if $a\in \R^D$ and ind$(a)=k$, then $a^k\in\R^\sharp$,
$(a^k)^\sharp=(a^d)^k$ and $p_{a^k}=1-aa^d$. Furthermore, according to \cite[Theorem 11]{M},
$a\in\R$ is Drazin invertible if and only if $a$ is invertible along $a^m$, for some $m\in\mathbb{N}$.
Moreover, in this case, $a^d=a^{\parallel a^m}$.\par

\indent Next the definition of generalized Drazin invertible elements will be recalled. 
However, to this end some preparation is needed.\par

\indent Given $\R$ a unitary ring, an element $a\in\R$ is said to be \it quasinilpotent\rm,
if for every $x\in$ comm$(a)$, $1+xa\in\R^{-1}$, where comm$(a)=\{x\in\R\colon ax=xa\}$
(see \cite[Definition 2.1]{KP}). The set of all quasinilpotent elements of $\R$ will be denoted by $\R^{qnil}$.
Note that if $\R^{nil}$ denotes the set of nilpotent elements of $\R$, then  $\R^{nil}\subset \R^{qnil}$
(see \cite{KP}).\par

\indent Recall that $a\in\R$ is said to be \it generalized Drazin invertible, \rm if there exists $b\in\R$ such that
$$
b\in\hbox{\rm comm}^2(a), \hskip.3truecm ab^2=b,\hskip.3truecm a^2b-b\in\R^{qnil},
$$
where $\hbox{\rm comm}^2(a)=\{x\in\R\colon xy=yx \hbox{ for all } y\in \hbox{\rm comm}(a)\}$
(see \cite[Definition 4.1]{KP}). The set of all generalized Drazin invertible elements of $\R$ 
will be donted by $\R^{gD}$. Note that according to \cite[Theorem 4.2]{KP},
necessary and sufficient for $a\in \R^{gD}$ is that there exists $p\in \hbox{\rm comm}^2(a)\cap\R^\bullet$
such that $ap\in\R^{qnil}$ and $a+p\in\R^{-1}$. Moreover, since according to \cite[Proposition 2.3]{KP} this idempotent is unique,
it is called the \it spectral idempotent \rm of $a$ and it is denoted by $a^\pi$. Furthermore, according to \cite[Theorem 4.2]{KP},
$a\in\R^{gD}$ has at most one generalized Drazin inverse, which will be denoted by $a^D$. In addition, in this case,
$a^\pi=1-aa^D=1-a^Da$ (\cite[Page 142]{KP}. \par

\indent On the other hand, note that $\R^\sharp\subset\R^D\subset \R^{gD}$. Moreover, if $a\in\R^D$, then
$a^D=a^d$ and $a^\pi=1-aa^d$ (\cite[Proposition 4.9 and Remark 4.10]{KP}), In particular, if $a\in\R^\sharp$, then $p_a=a^\pi$. 
Since the group inverse has a key role in this article, the spectral idempotent 
of a group invertible elment $a\in\R$ is denoted by $p_a$.\par
\indent Recall that according to \cite[Theorem 8]{M2},
if $a\in \R^{gD}$, then $a$ is invertible along $d=1-a^\pi$ and $a^{\parallel 1-a^\pi}=a^D$.\par

\indent The last generalized inverse that will be recalled in this section is the Moore-Penrose inverse.
\par

\indent Let $\R$ be a unitary ring. An involution $\hbox{}^*\colon\R\to \R$ is an anti-isomorphism:
$$
(a+b)^*=a^*+b^*, \hskip.3truecm (ab)^*=b^*a^*,\hskip.3truecm  (a^*)^*=a,
$$
where $a$, $b\in \R$.\par
\indent An element $a\in\R$ is said to be \it Moore-Penrose invertible, \rm if there exists a (necessarily unique) $b\in \R$
such that 
$$
aba=a, \hskip.3truecm bab=b, \hskip.3truecm (ab)^*=ab, \hskip.3truecm (ba)^*=ba. 
$$
The Moore-Penrose inverse of $a$ is denoted by $a^\dag$ and the set of all Moore-Penrose
invertible elements of $\R$ will be denoted by $\R^\dag$  (see for example \cite{KP}).
Recall that according to \cite[Theorem 11]{M}, necessary and sufficient for $a\in\R^\dag$ is that
$a$ is invertible along $a^*$. Moreover, in this case $a^{\parallel a^*}=a^\dag$.\par

\indent Finally, recall that $a\in\R^\dag$ is said to be EP, if $aa^\dag=a^\dag a$ (see \cite{KP}).
Let $\R^{EP}$ be the set of all EP elements in $\R$.
 
\section{Equivalent conditions for the inverse along an element}

\noindent In this section new conditions equivalent to the ones in Definition \ref{def2} will be given.\par

\indent First of all note that if $\R$ is a ring and $b$, $d\in \R$ are such that $b\R=d\R$, then $b_{-1}(0)=d_{-1}(0)$; similarly from $\R b=\R d$ it can be derived that 
$b^{-1}(0)=d^{-1}(0)$. These conditions will be used to prove the invertibility along an element. Next follows a preliminary result. \par

\begin{pro}\label{pro3}Let $\R$ be a unitary ring and consider $b$, $d\in\R$ .\par
\noindent \rm (a) \it Let $a\in\R$ be such that $b$ is an outer inverse of $a$. Then the following statements hold.\par
\noindent \rm (i) \it If $b^{-1}(0)\subseteq d^{-1}(0)$, then $d=dab$; in particular $\R d\subseteq \R b$.\par
\noindent \rm  (ii) \it If $b_{-1}(0)\subseteq d_{-1}(0)$, then $d=bad$; in particular $d\R\subseteq b\R$.\par
\noindent \rm (b) \it  Suppose that $d\in\hat{\R}$. Then the following statements hold.\par
\noindent \rm  (iii) \it If $d^{-1}(0)\subseteq b^{-1}(0)$, then $\R b \subseteq \R d$.\par
\noindent \rm  (iv) \it If $d_{-1}(0)\subseteq b_{-1}(0)$, then $b\R\subseteq d\R$.
\end{pro}
\begin{proof}(i). Since $b(1-ab)=0$, $1-ab \in b^{-1}(0) \subseteq d^{-1}(0)$.  Thus, $d(1-ab)=0$.\par
\noindent (ii). Apply a similar argument to the one used in the proof of statement (i). \par
\noindent (iii). Let $\overline{d}\in\R$ be such that $d\overline{d}d=d$. Then $d$ is an outer inverse
of $\overline{d}$. Therefore, according to statement (i), $\R b\subseteq \R d$.\par
\noindent (iv). Apply a similar argument  to the one used in the proof of statement (iii),
using statement (ii) instead of statement (i).
\end{proof}

\indent In the following theorems, equivalent conditions to the ones in Definition \ref{def2} will be proved.
These conditions will be presented in two different theorems to show when it is necessary to assume 
the regularity of the element $d\in \R$. \par

\begin{thm}\label{thm4}Let $\R$ be a unitary ring and consider $a$, $b$, $d\in\R$ be such that $b$ is an outer inverse of $a$.
Then, the following statements are equivalent.\par
\noindent \rm (i) \it $b$ is the inverse of $a$ algong $d$.\par
\noindent\rm (ii) \it $\R d=\R b$, $b\R\subseteq d\R$ and $b_{-1} (0)\subseteq d_{-1}(0)$.\par
\noindent\rm (iii) \it $b\R=d R$, $\R b\subseteq \R d$ and $b^{-1}(0)\subseteq d^{-1}(0)$.\par
\noindent\rm (iv) \it $\R b\subseteq \R d$, $b\R\subseteq d\R$,   $b_{-1} (0)\subseteq d_{-1}(0)$
and $b^{-1}(0)\subseteq d^{-1}(0)$.
\end{thm}
\begin{proof} It is clear that statement (i) implies all the other statements.\par
\indent On the other hand, to prove that statements (ii)-(iv) imply that $b$ is the inverse of $a$
along $d$, note that according to Proposition \ref{pro3} (i) (respectively  Proposition \ref{pro3} (ii)), if $b^{-1}(0)\subseteq d^{-1}(0)$ 
(respectively if $b_{-1}(0)\subseteq d_{-1}(0)$), then $\R d\subseteq \R b$
(respectively $d\R\subseteq b\R$).
\end{proof}

\begin{thm}\label{thm5}Let $\R$ be a unitary ring and consider $a$, $b$, $d\in\R$ be such that $b$ is an outer inverse of $a$
and $d\in\hat{\R}$.
Then, the following statements are equivalent.\par
\noindent \rm (i) \it $b$ is the inverse of $a$ along $d$.\par
\noindent\rm (ii) \it $\R d=\R b$, $d\R\subseteq b\R$ and $d_{-1} (0)\subseteq b_{-1}(0)$.\par
\noindent\rm (iii) \it $b\R=d R$, $\R d\subseteq \R b$ and $d^{-1}(0)\subseteq b^{-1}(0)$.\par
\noindent\rm (iv) \it $\R b= \R d$,   $b_{-1} (0)= d_{-1}(0)$.\par
\noindent \rm (v) \it $\R b\subseteq \R d$,  $d\R\subseteq b\R$, $b^{-1}(0)\subseteq d^{-1}(0)$
and $d_{-1} (0)\subseteq b_{-1}(0)$.\par
\noindent \rm (vi) \it $\R d\subseteq \R b$, $b\R\subseteq d\R$, $b_{-1} (0)\subseteq d_{-1}(0)$
and $d^{-1}(0)\subseteq b^{-1}(0)$.\par
\noindent \rm (vii) \it $\R d\subseteq\R b$, $d\R\subseteq b\R$, $d^{-1}(0)\subseteq b^{-1}(0)$ and
$d_{-1} (0)\subseteq b_{-1}(0)$.\par
\noindent \rm (viii)  \it $b\R=d R$, $b^{-1}(0)= d^{-1}(0)$.\par
\noindent \rm (ix) \it $\R b\subseteq \R d$,  $b^{-1}(0)\subseteq d^{-1}(0)$ and $b_{-1} (0)= d_{-1}(0)$.\par
\noindent \rm (x) \it  $\R d\subseteq\R b$, $d^{-1}(0)\subseteq b^{-1}(0)$ and $b_{-1} (0)= d_{-1}(0)$.\par
\noindent \rm (xi) \it $d\R\subseteq b\R$, $d_{-1} (0)\subseteq b_{-1}(0)$ and $b^{-1} (0)= d^{-1}(0)$.\par
\noindent \rm (xii) \it $b\R\subseteq d\R$, $b_{-1} (0)\subseteq d_{-1}(0)$ and $b^{-1} (0)= d^{-1}(0)$.\par
\noindent \rm (xiii) \it  $b^{-1} (0)= d^{-1}(0)$ and $b_{-1} (0)= d_{-1}(0)$.\par
\end{thm}
\begin{proof} As in the proof of Theorem \ref{thm4}, statement (i) implies all the other statements.\par
\indent To prove that statements (ii)-(xiii) imply that $b$ is the inverse of $a$ along $d$,
note the following facts. According to Proposition \ref{pro3} (i) (respectively  Proposition \ref{pro3} (ii)), if $b^{-1}(0)\subseteq d^{-1}(0)$ 
(respectively if $b_{-1}(0)\subseteq d_{-1}(0)$), then $\R d\subseteq \R b$
(respectively $d\R\subseteq b\R$). On the other hand, 
according to Proposition \ref{pro3} (iii) (respectively  Proposition \ref{pro3} (iv)), if $d^{-1}(0)\subseteq b^{-1}(0)$ 
(respectively if  $d_{-1}(0)\subseteq b_{-1}(0)$), then $\R b \subseteq \R d$
(respectively $b\R\subseteq d\R$).
\end{proof}

\section{ Further invertible elements }

\noindent In this section, using ideas similar to the ones in \cite[Theorem 7]{M}, more equivalent conditions to the existence of an inverse along an element will be given. What is more, thanks to these characterizations, given a unitary ring $\R$, $a\in\R$ and $d\in\hat{\R}$ such that $a$ is invertible along $d$, more elements invertible along $d$ will be constructed.\par

\indent First of all, note that  if $d\in\hat{\R}$ and $\overline{d}$ is a generalized inverse of $d$, then 
$d\overline{d}\R d\overline{d}$ (respectively $\overline{d}d \R \overline{d}d$) is a subring of $\R$
with identity $d\overline{d}$ (respectively $\overline{d}d$). \par

\begin{thm}\label{thm8}Let $\R$ be a unitary ring and consider $a\in\R$ and $d\in\hat{\R}$.
Then, if $\overline{d}$ is a generalized inverse of $d$, the following conditions are equivalent.\par
\noindent \rm (i) \it The element $a$ is invertible along $d$.\par
\noindent \rm (ii) \it $dad\overline{d}\in (d\overline{d}\R d\overline{d})^{-1}$.\par
\noindent \rm (iii) \it $\overline{d}dad\in (\overline{d}d\R \overline{d}d)^{-1}$.\par
\noindent \rm (iv) \it  $dad\overline{d}$ is group invertible and $p_{dad\overline{d}}= 1-d\overline{d}$.\par
\noindent \rm (v) \it  $\overline{d}dad$ is group invertible and $p_{\overline{d}dad}= 1-\overline{d}d$.\par
\noindent In addition, if $x\in d\overline{d}\R d\overline{d}$  (respectively $y \in \overline{d}d\R \overline{d}d$)
 is an inverse of $dad\overline{d}$ (respectively  of $\overline{d}dad$), then $xd=a^{\parallel d}$
(respectively $dy= a^{\parallel d}$).
\end{thm}
\begin{proof}Suppose that $a^{\parallel d}$ exists. To prove statement (ii), first note that $dad\overline{d}=d\overline{d} dad\overline{d}
\in d\overline{d}\R d\overline{d}$. Moreover, according to \cite[Theorem 7]{M}, $d\overline{d}\R=d\R= da\R$. In particular,
there is $u\in \R$ such that $d\overline{d}=dau$. Furthermore, according to  \cite[Theorem 7]{M},
$da$ is group invertible. As a result,
$$
d\overline{d}=dau=da (da (da)^\sharp ud\overline{d}).
$$
\noindent Thus, $z= da (da)^\sharp ud\overline{d}=d\overline{d}da (da)^\sharp ud\overline{d}\in d\overline{d}\R d\overline{d}$ and
$d\overline{d}=daz=dad\overline{d}z$.\par

\indent On the other hand, according again to  \cite[Theorem 7]{M}, $\R d=\R ad$ and 
$ad$ is group invertible. In particular, there is $v\in \R$ such that $d=vad$. Thus,
$$
d\overline{d}=vad\overline{d}= (d\overline{d}v(ad)^\sharp ada)dad\overline{d}= (d\overline{d}v(ad)^\sharp adad\overline{d})dad\overline{d}.
$$
Then, $w=d\overline{d}v(ad)^\sharp adad\overline{d}\in d\overline{d}\R d\overline{d}$
and $d\overline{d}=wdad\overline{d}$.  Consequently, $dad\overline{d}\in  (d\overline{d}\R d\overline{d})^{-1}$.\par
\indent Suppose that stetement (ii) holds. Let $x\in d\overline{d}\R d\overline{d}$ such that 
$$
dad\overline{d} x=d\overline{d}=xdad\overline{d}.
$$
Then, it will be proved that $xd=a^{\parallel d}$. First of all, note that
$$
xd a xd=xdad\overline{d}xd=d\overline{d}xd=xd.
$$
\noindent In addition, clearly $\R xd\subseteq \R d$ and since $x=d\overline{d}x$, $xd\R\subseteq d\R$. 
Moreover, since $xdad\overline{d}=d\overline{d}$, $xdad=d$, which implies that $d\R\subseteq xd\R$.
Similarly, since  $dad\overline{d} x=d\overline{d}$, $dad\overline{d} xd=d$, which implies that $\R d\subseteq \R xd$.\par

\indent To prove the equivalence between statements (ii) and (iv), apply Remark \ref{rema13} (ii).

\indent The equivalences among the statements (i), (iii) and (v) can be proved using similar arguments.
\end{proof}

\begin{rema}\label{rema4000}\rm Under the same hypotheses in Theorem \ref{thm8} and using the same notation in this Theorem, note that
since $xd=a^{\parallel d}$, $(dad\overline{d})^{-1}_{d\overline{d}\R d\overline{d}}=(dad\overline{d})^\sharp=x=a^{\parallel d}\overline{d}$.
Similarly, since $dy=a^{\parallel d}$, $(\overline{d}dad)^{-1}_{\overline{d}d\R \overline{d}d}=(\overline{d}dad)^\sharp=y= \overline{d}a^{\parallel d}$.\par
\end{rema}

\indent Given a unitary ring $\R$, $a\in\R$ and $d\in\hat{\R}$ such that $a$ is invertible along $d$, applying Theorem \ref{thm8} new elements 
invertible along $d$ can be created.\par

\begin{cor}\label{cor9}Let $\R$ be a unitary ring and consider $a\in\R$ and $d\in\hat{\R}$. 
Then, if $\overline{d}$ is a generalized inverse of $d$, the following statements are equivalent.\par
\noindent \rm (i) \it $a$ is invertible along $d$.\par
\noindent \rm (ii) \it $ad\overline{d}$ is invertible along $d$.\par
\noindent \rm (iii) \it $\overline{d}da$ is invertible along $d$.\par
\noindent Furthermore, in this case,
$$
a^{\parallel d}=(ad\overline{d})^{\parallel d}=(\overline{d}da)^{\parallel d}.
$$
\end{cor}
\begin{proof}Note that $dad\overline{d}d\overline{d}=dad\overline{d}$.
Therefore, according to Theorem \ref{thm8}, statements (i) and (ii) are equivalent.\par
\indent A similar argument proves the equivalence between statements (i) and (iii).\par
\indent According to Theorem \ref{thm8}, if $x\in d\overline{d}\R d\overline{d}$ is an inverse of $dad\overline{d}$,
then $a^{\parallel d}=xd$. In particular, since $dad\overline{d}d\overline{d}=dad\overline{d}$,
$a^{\parallel d}=(ad\overline{d})^{\parallel d}$. A similar argument,
using the inverse of $\overline{d}dad\in \overline{d}d\R \overline{d}d$, proves that 
$a^{\parallel d}=(\overline{d}da)^{\parallel d}$.
\end{proof}

\begin{cor}\label{cor10}Let $\R$ be a unitary ring and consider $a\in\R$ and $d\in\hat{\R}$. 
Then, if $\overline{d}$ is a generalized inverse of $d$, $a^{\parallel d}$ exists and $x$, $y\in\R$, the following statements hold.\par
\noindent \rm (i) \it $a+x(1-d\overline{d})$ is invertible along $d$.\par
\noindent \rm (ii) \it  $a+(1-\overline{d}d)y$ is invertible along $d$.\par
\noindent \rm (iii) \it \noindent  $a+x(1-d\overline{d})+(1-\overline{d}d)y$ is invertible along $d$.\par
\noindent Moreover, $a^{\parallel d}=(a+x(1-d\overline{d}))^{\parallel d}=(a+(1-\overline{d}d)y)^{\parallel d}
= (a+x(1-d\overline{d})+(1-\overline{d}d)y)^{\parallel d}$.
\end{cor}
\begin{proof} Since $ad\overline{d}=(a+x(1-d\overline{d}))d\overline{d}$, according to  Corollary \ref{cor9}, 
$(a+x(1-d\overline{d}))^{\parallel d}$ exists and  $a^{\parallel d}=(a+x(1-d\overline{d}))^{\parallel d}$.
A similar argument proves that $a+(1-\overline{d}d)y$ is invertible along $d$ and $a^{\parallel d}=(a+(1-\overline{d}d)y)^{\parallel d}$ .
Statement (iii) and the remaining identity can be derived applying statements (i) and (ii).
\end{proof}

\indent To end this section, the case $d\in\R^\bullet$ will be considered.\par

\begin{cor}\label{cor1000}Let $\R$ be a unitary ring and consider $a\in\R$ and $p\in\R^\bullet$. 
Then, the following statements are equivalent.\par
\noindent \rm (i) \it  $a$ is invertible along $p$.\par
\noindent \rm (ii) \it $ap$ is invertible along $p$.\par
\noindent \rm (iii) \it  $pa$ is invertible along $p$.\par
\noindent \rm (iv) \it  $pap\in (p\R p)^{-1}$.\par
\noindent \rm (v) \it  $pap$ is group invertible and $p_{pap}=1-p$.\par
\noindent Furthermore, in this case,
$$
a^{\parallel p}=(ap)^{\parallel p}=(pa)^{\parallel p}=(pap)^\sharp.
$$
\end{cor} 
\begin{proof}Apply Theorem \ref{thm8} and Corollary \ref{cor9}.
\end{proof}

\section{The set of invertible elements  along a fixed $d\in\hat{\R}$}

\indent In this section, given a unitary ring $\R$ and a regular element $d\in\R$,
the set of all invertible elements along $d$ will be fully characterized. Moreover, some special
cases will be also considered. To this end, the set under consideration will
be denoted by $\R^{\parallel d}$, i.e., $\R^{\parallel d}=\{a\in \R\colon  a^{\parallel d}\hbox{ exits}\}$. Note that if $\tilde{d}\in\hat{R}$
is such that $d\R=\tilde{d}\R$ and $\R d=\R \tilde{d}$, then $\R^{\parallel d}=\R^{\parallel\tilde{d}}$.
Next follows the main
theorem of this section.\par

\begin{thm}\label{thm11}Let $\R$ be a unitary ring and consider $d$ and  $\overline{d}\in \R$
such that $d\in\hat{\R}$ and $\overline{d}$ is  a generalized inverse of $d$.  \par
\noindent \rm (i) \it Then, the following identity holds:\par
$$
\R^{\parallel d}=\overline{d}(d\overline{d}\R d\overline{d})^{-1}+(1-\overline{d}d)\R d\overline{d}+\R(1-d\overline{d}).
$$ 
Moreover, given $x$, $y\in\R$ and $v\in (d\overline{d}\R d\overline{d})^{-1}$, then, if $w\in d\overline{d}\R d\overline{d}$
is such that $vw=wv=d\overline{d}$,
$$
(\overline{d}v)^{\parallel d}=(\overline{d}v + (1-\overline{d}d)x d\overline{d}+y(1-d\overline{d}))^{\parallel d}= wd.
$$
\noindent \rm (ii) In addition, \par
$$
\R^{\parallel d}=(\overline{d}d\R \overline{d}d)^{-1}\overline{d}+\overline{d}d\R(1- d\overline{d})+(1-\overline{d}d)\R.
$$ 
\noindent Furthermore, given $s$, $t\in\R$ and $z\in (\overline{d}d\R \overline{d}d)^{-1}$, then, if $u\in \overline{d}d\R \overline{d}d$
is such that $zu=uz=\overline{d}d$,
$$
(z\overline{d})^{\parallel d}=(z\overline{d} + \overline{d}ds (1- d\overline{d})+(1-d\overline{d})t)^{\parallel d}= du.
$$
\noindent (iii) In particular, if $p\in\R^\bullet\setminus\{0,1\}$, then\par
$$
\R^{\parallel p}=(p\R p)^{-1}+p\R (1-p)+(1-p)\R p+(1-p)\R (1-p),
$$ 
\noindent and if $r\in (p\R p)^{-1}$, $m\in (p\R (1-p)+(1-p)\R p+(1-p)\R (1-p))$, and $l\in (p\R p)^{-1}$ is such that 
$rl=lr=p$, then 
$$
r^{\parallel p}=(r+m)^{\parallel p}=l.
$$
\end{thm}
\begin{proof}(i). Let $v\in (d\overline{d}\R d\overline{d})^{-1}$ and consider $a=\overline{d}v$. Thus, 
$dad\overline{d}=d\overline{d}vd\overline{d}=v$. Consequently, according to Theorem \ref{thm8},
$a^{\parallel d}$ exists and $a^{\parallel d}=wd$. Moreover, according to Corollary \ref{cor10},
$(a+ (1-\overline{d}d)x d\overline{d}+y(1-d\overline{d}))^{\parallel d}$ exists and $a^{\parallel d}=
 (a+ (1-\overline{d}d)x d\overline{d}+y(1-d\overline{d}))^{\parallel d}$. \par
\indent To prove the converse, let $a\in \R$ such that  $a^{\parallel d}$ exists.
Then
$$
a=\overline{d}dad\overline{d}+(1-\overline{d}d)ad\overline{d}+a(1-d\overline{d}).
$$
However, $\overline{d}dad\overline{d}= \overline{d}(dad\overline{d})$ and, according to Theorem \ref{thm8},
$dad\overline{d}\in (d\overline{d} \R d\overline{d})^{-1}$.\par
\noindent (ii). Apply a similar argument to the one used in the proof of statement (a).\par
\noindent (iii). Given $d=p\in\R^\bullet\setminus\{0,1\}$, consider $\overline{d}=p$. Then apply what has been prove. 
\end{proof}

\begin{rema}\label{rema12}\rm Under the same hypotheses in Theorem \ref{thm11} and using the same notation in this Theorem,
note the following facts.\par
 \begin{align*}
\noindent \hbox{\rm (i).} (1-\overline{d}d)\R d\overline{d}+\R(1-d\overline{d})&= \overline{d}d\R (1- d\overline{d})+(1-\overline{d}d)\R\\
&=(1-\overline{d}d)\R d\overline{d}+\overline{d}d \R(1-d\overline{d}) + (1-\overline{d}d) \R(1-d\overline{d}).\\
\end{align*}

\noindent (ii). $\overline{d}(d\overline{d}\R d\overline{d})^{-1}\subseteq \overline{d}d\R d\overline{d}$.
In addition, 
\begin{align*}
&\overline{d}(d\overline{d}\R d\overline{d})^{-1}\cap ((1-\overline{d}d)\R d\overline{d}+\R(1-d\overline{d}))=\emptyset,\\
&(1-\overline{d}d)\R d\overline{d}\cap\R(1-d\overline{d})=0, \, \overline{d}d\R(1-d\overline{d}\cap (1-\overline{d}d)\R(1-d\overline{d})=0.
\end{align*}

\noindent (iii). In particular, given $p\in\R^\bullet\setminus\{0,1\}$, 
\begin{align*}
&\hskip.7truecm (p\R p)^{-1}\cap (p\R (1-p)+(1-p)\R p+(1-p)\R (1-p))=\emptyset,\\
&\hskip.7truecm p\R (1-p)\cap  (1-p)\R p=p\R (1-p)\cap (1-p)\R (1-p)= \\
&\hskip.7truecm  (1-p)\R p\cap (1-p)\R (1-p)=0.\\
\end{align*}
\noindent (iv). The elements $w$, $u$ and $l$ in Theorem \ref{thm11} can be presented as in Remark \ref{rema4000} ($w$ and $u$) 
and Corollary \ref{cor1000} ($l$). For example, $w= (\overline{d}v)^{\parallel d}\overline{d}$, $u=\overline{d}(z\overline{d})^{\parallel d}$
and $l=r^\sharp$.
\end{rema}

\indent Applying Theorem \ref{thm11} it is possible to give another characterization of the inverse along an element.\par

\begin{thm}\label{thm19}Let $\R$ be a unitary ring and consider $d\in\hat{R}$. Then, if $\overline{d}$ is a generalized inverse
of $d$ and $a\in\R$, the following statements are equivalent.\par
\noindent \rm (i) \it The element $a$ is invertible along $d$.\par
\noindent \rm (ii) \it There exist unique $s$, $t\in \R$ such that $a=\overline{d} s +t$, $s\in\R^\sharp$, $p_s=1- d\overline{d}$
and $\overline{d}dtd\overline{d}=0$. In addition, $a^{\parallel d}=s^\sharp d$.\par
\noindent \rm (iii) \it There exist unique $u$, $v\in \R$ such that $a=u\overline{d}+v$, $u\in \R^\sharp$, $p_u=1-\overline{d}d$
and $\overline{d}d vd\overline{d}=0$. Moreover, $a^{\parallel d}=du^\sharp$.\par
\noindent In particular, if $p\in\R^\bullet$, necessary and sufficient for $a\in\R^p$ is that
there exist unique $s$, $t\in \R$ such that $a=s+t$, $s\in\R^\sharp$, $p_s=1-p$ and $ptp=0$. Furthermore, 
$a^{\parallel p}=s^\sharp$.	 
\end{thm}
\begin{proof}If $a$ is invertible along $d$, then according to Theorem \ref{thm11} (i), there exist $s\in\ (d\overline{d}\R d\overline{d})^{-1}$
and $t\in  (1-\overline{d}d)\R d\overline{d}+\R(1-d\overline{d})$ such that $a=\overline{d}s+t$. Note that according to Remark \ref{rema13} (ii) and Theorem \ref{thm11} (i), $s\in\R^\sharp$, $p_s=1- d\overline{d}$ and $a^{\parallel d}=s^\sharp d$. Moreover, $\overline{d}dtd\overline{d}=0$.\par

\indent On the other hand, if statement (ii) holds, then according to Remark \ref{rema13} (ii), $s\in (d\overline{d}\R d\overline{d})^{-1}$
and $t\in (1-\overline{d}d)\R d\overline{d}+\R(1-d\overline{d})$. Thus, according to Theorem \ref{thm11} (i), $a\in\R^{\parallel d}$.\par

\indent Let $s_1$, $s_2$ and $t_1$ and $t_2$ be such that $a=\overline{d}s_1 +t_1=\overline{d}s_2 +t_2$, $s_1$, $s_2\in
(d\overline{d}\R d\overline{d})^{-1}$ and $t_1$, $t_2\in   (1-\overline{d}d)\R d\overline{d}+\R(1-d\overline{d})$. Then, 
$t_1=t_2$ and $\overline{d}s_1=\overline{d}s_2$. However, multiplying by $d$ on the left side, $s_1=s_2$.\par

\indent Similar arguments prove the equivalence between statements (i) and (iii), applying in particular Theorem \ref{thm11} (ii)
instead of Theorem \ref{thm11} (i).  \par
\indent The last statement can be proved using similar arguments and applying Theorem \ref{thm11} (iii).
\end{proof}
\indent Next the particular cases of group, (generalized) Drazin and commuting Moore-Penrose invertible elements  will be considered.\par

\begin{thm}\label{thm14}\rm (a) \it Let $\R$ be a unitary ring and consider $a\in\R$.\par
\noindent \rm (i) \it If $a\in \R^\sharp$ and $n\in\mathbb N$, then
$$
\R^{\parallel a^n}=\R^{\parallel (a^\sharp)^n}=\R^{\parallel 1-p_a}.
$$
\noindent In addition, if $x$ belongs to one of these sets, then
$x^{\parallel a^n}=x^{\parallel (a^\sharp)^n}=x^{\parallel 1-p_a}$.\par
\noindent \rm (ii) \it If $a\in\R^D$, \rm ind\it$(a)=k$, $n\in\mathbb N$, $1\le j\le k-1$ and
$m\ge k$, then
$$ 
\R^{\parallel (a^d)^n}=\R^{\parallel a^m}=\R^{\parallel a^ja^da}=\R^{\parallel 1-a^\pi}.
$$
\noindent Moreover, if $x$ belongs to one of these sets, then
$x^{\parallel (a^d)^n}=x^{\parallel a^m}=x^{\parallel a^ja^da}=x^{\parallel 1-a^\pi}$.\par
\noindent \rm (iii) \it If $a\in\R^{gD}$ and $n\in\mathbb N$, then 
$$ 
\R^{\parallel (a^D)^n}=\R^{\parallel a^na^da}=\R^{\parallel 1-a^\pi}.
$$
\noindent Furthermore, if $x$ belongs to one of these sets, then
$x^{\parallel (a^d)^n}=x^{\parallel a^na^da}=x^{\parallel 1-a^\pi}$.\par
\noindent \rm (b) \it Let $\R$ be a unitary ring with involution.\par
\noindent \rm (iv) \it If $a\in\R$ is EP and $n\in \mathbb N$, then 
$$
\R^{\parallel a^n}=\R^{\parallel (a^\dag)^n}=\R^{\parallel (a^*)^n}=\R^{\parallel ((a^\dag)^*)^n}=\R^{\parallel aa^\dag}.
$$
\noindent If $x$ belongs to one of these sets, then
$x^{\parallel a^n}=x^{\parallel (a^\dag)^n}=x^{\parallel (a^*)^n}=x^{\parallel ((a^\dag)^*)^n}=x^{\parallel aa^\dag}$.
\end{thm}
\begin{proof}(i). According to Theorem \ref{thm11},
\begin{align*}
\R^{\parallel a} &= a^\sharp ((1-p_a)\R (1-p_a))^{-1}+p_a\R (1-p_a)+\R p_a,\\
\R^{\parallel a^\sharp} &= a ((1-p_a)\R (1-p_a))^{-1}+p_a\R (1-p_a)+\R p_a,\\
\R^{\parallel 1-p_a} &= ((1-p_a)\R (1-p_a))^{-1}+p_a\R (1-p_a)+\R p_a.\\
\end{align*}

\noindent However, since $a$ and $a^\sharp\in ((1-p_a)\R (1-p_a))^{-1}$ (Remark \ref{rema13} (ii)),
 $$
a^\sharp ((1-p_a)\R (1-p_a))^{-1}=a ((1-p_a)\R (1-p_a))^{-1}=((1-p_a)\R (1-p_a))^{-1}.
$$
\noindent Thus, $\R^{\parallel a}=\R^{\parallel a^\sharp}=\R^{\parallel 1-p_a}$. In addition, since according to Remark \ref{rema13} (iv),
$a^n\in\R^\sharp$, $(a^n)^\sharp=(a^\sharp)^n$ and $p_{a^n}=p_a$, applying what has been proved to $a^n$,
$\R^{\parallel a^n}=\R^{\parallel (a^\sharp)^n}=\R^{\parallel 1-p_a}$ ($n\in\mathbb N$).\par

\indent Note that to prove the remaining statement, it is enoug to consider  the case $a$ group invertible.
Let $a\in\R^\sharp$ and $x\in\R^{\parallel a}$. According to Theorem \ref{thm11} (i) applied to $\overline{d}=a^\sharp$, $x^{\parallel a}=wa$, where
$w\in (aa^\sharp\R aa^\sharp)^{-1}$ is such that $waxaa^\sharp= axaa^\sharp w=aa^\sharp$. On the other hand,
according to Theorem \ref{thm11} (iii) applied to $aa^\sharp$, $x^{\parallel aa^\sharp}= v$, where $v\in (aa^\sharp\R aa^\sharp)^{-1}$ is 
such that $v aa^\sharp x aa^\sharp=aa^\sharp xaa^\sharp v=aa^\sharp$. Now well, since $a=aaa^\sharp$
$$
waaa^\sharp  xaa^\sharp =v aa^\sharp x aa^\sharp.
$$
However, since $aa^\sharp x aa^\sharp\in (aa^\sharp\R aa^\sharp)^{-1}$, 
$$
x^{\parallel a}=wa=v=x^{\parallel aa^\sharp}.
$$
\noindent (ii). Recall that acording to Remark \ref{rema13} (iii). $a^d\in\R^\sharp$, $(a^d)^\sharp=aa^da$ and
$p_{a^d}= a^\pi$. In addition, if $1\le j\le k-1$ and $m\ge k$ ($k=$ ind$(a)$), then it is not difficult to prove that 
$((a^d)^\sharp)^{j}=a^ja^da$ and $((a^d)^\sharp)^m=a^m$. To prove statement (ii), apply statement (i) to $a^d$.\par

\noindent (iii). As in the proof of statement (ii), acording to Remark \ref{rema13} (iii). $a^D\in\R^\sharp$, $(a^D)^\sharp=aa^Da$ and
$p_{a^D}= a^\pi$. Moreover, if $n\in\mathbb N$, then  $((a^D)^\sharp)^n=a^na^Da$. Then, apply statement (i) to $a^D$. \par

\noindent (iv). If $a\in\R$ is EP, then $a$, $a^\dag$, $a^*$ and $(a^\dag)^*\in \R^\sharp$. Note that $a^\sharp=a^\dag$ and 
$(a^*)^\sharp=(a^\dag)^*$. Moreover, $p_a=p_{a^\dag}=p_{a^*}=p_{(a^\dag)^*}=1-aa^\dag$. Then, apply statement (i).
\end{proof}

\indent Next a characterization of invertible elements along a group invertible element will be considered.\par

\begin{cor}\label{cor21}Let $\R$ be a unitary ring and consider $d\in\R^\sharp$ and $a\in\R$. Then, the following statements
are equivalent.\par
\noindent \rm (i) \it The element $a$ is invertible along $d$.\par
\noindent \rm (ii) \it There exist unique $s$, $t\in \R$ such that $a=s+t$, $s\in\R^\sharp$, $p_s=p_d$ and $(1-p_d)t(1-p_d)=0$.
Moreover,  $a^{\parallel d}= s^\sharp$.
\end{cor}
\begin{proof}Apply Theorem \ref{thm14} (i) and Theorem \ref{thm19}.
\end{proof}

\begin{rema}\label{rema18}\rm Let $\R$ be a unitary ring. Clearly, if $d\in \R^{-1}\cup\R^\bullet$, then Corollary \ref{cor21} applies to $d$.
In addition, according to the proof of  Theorem \ref{thm14}, Corollary \ref{cor21} applies to $d=x^n$ ($x\in\R^\sharp$, $n\in\mathbb N$), to $d=x^nx^Dx$ ($x\in\R^{gD}$, $n\in\mathbb N$) or to $d=y^j$ ($y\in\R^D$, $j\in\mathbb N$, $j\ge \hbox{\rm ind}(y)$). Moreover, if $\R$ is a unital ring with an involution,
according again to Theorem \ref{thm14}, 
Corollary \ref{cor21} applies to $d= x^n$, $d=(x^\dag)^n$,  $d =(x^*)^n$ or  $d= ((x^*)^\dag)^n$ ($x\in \R$ an EP element, $n\in\mathbb N$).
\end{rema}

\indent Given a unitary ring $\R$ and $d\in \hat{\R}$, $a$, $b\in\R^{\parallel d}$ will be said \it to satisfy the reverse order law, \rm
if $ab\in \R^{\parallel d}$ and $(ab)^{\parallel d}= b^{\parallel d}a^{\parallel d}$. In the following theorem the elements $d\in\hat{\R}$ such that  the reverse order law holds for every pair of elements of $\R^{\parallel d}$ will be characterized. Note that the next Theorem also provides a 
characterization of group invertible elements.\par

\begin{thm}\label{thm15}Let $\R$ be a unitary ring and consider $d\in \hat{R}$.
Then, the following statements are equivalent.\par
\noindent \rm (i) \it The reverse order law holds for   any $a$ and $b\in\R^{\parallel d}$.\par
\noindent \rm (ii) \it The element $d$ is group invertible.\par
\end{thm}
\begin{proof}Let $\overline{d}\in\R$ such that $d=d\overline{d} d$. According to  Theorem \ref{thm11} (i), 
$\overline{d}d\overline{d}\in\R^{\parallel d}$, what is more, $(\overline{d}d\overline{d})^{\parallel d}=d$.\par
\indent Suppose that statement (i) holds and consider $a=\overline{d}d\overline{d}=b$. Then, $ab\in\R^{\parallel d}$ and according to
Theorem \ref{thm11} (i)-(ii), there exist $x\in\ (d\overline{d}\R d\overline{d})^{-1}$ and $u\in (\overline{d}d \R \overline{d}d)^{-1}$
such that $(ab)^{\parallel  d}=yd=dv$, where $y\in (d\overline{d}\R d\overline{d})^{-1}$ and $v\in (\overline{d}d \R \overline{d}d)^{-1}$ are such that $xy=yx=d\overline{d}$ and $uv=vu=\overline{d}d$. 
As a result, 
$$
d^2=b^{\parallel d}a^{\parallel d}=(ab)^{\parallel d}= yd=dv.
$$
Therefore, $d=xd^2$ and $d=d^2u$. However, according to \cite[Theorem 4]{D},
$d\in\R^\sharp$.\par  

\indent To prove the converse, note that according to Theorem \ref{thm14} (i), it is enough to consider the case $p\in\R^\bullet$, $p\neq 0, 1$.
In addition, according to Theorem \ref{thm11} (iii), there exist $v$ and $w\in (p\R p)^{-1}$ and $s$ and $t\in  
p\R (1-p)+(1-p)\R p+(1-p)\R (1-p)$ such that $a=v+s$, $b=w+t$, $a^{\parallel  p}=l$ and $b^{\parallel p }=m$,
where $vl=lv=p$ and $wm=mw=p$. However, a direct calculation shows
that there is $z\in p\R (1-p)+(1-p)\R p+(1-p)\R (1-p)$ such that $ab=vw+z\in\R^{\parallel p}$. Furthermore,
according to Theorem \ref{thm11} (iii), $(ab)^{\parallel  p}=ml=b^{\parallel  p}a^{\parallel  p}$.
\end{proof}

\indent In the following theorem, given a unitary ring $\R$ and $d\in\hat{\R}$, the invertibility along
elements related to $d$ will be studied.\par

\begin{thm}\label{thm16}Let $\R$ be a unitary ring and consider $d\in\hat{\R}$. Then, if $a\in\R$ and $u\in\R^{-1}$,
the following statements are equivalents.\par
\noindent \rm (i) \it  $a\in\R^{\parallel d}$.\par
\noindent \rm (ii) \it  $au^{-1}\in\R^{\parallel ud}$.\par
\noindent \rm (iii) \it  $u^{-1}a\in\R^{\parallel du}$.\par
\noindent Moreover, $\R^{\parallel ud}=\R^{\parallel d}u^{-1}$, $\R^{\parallel du}=u^{-1}\R^{\parallel d}$ and if one of the statements
holds, then $(au^{-1})^{\parallel ud}= u a^{\parallel d}$ and $(u^{-1}a)^{\parallel du}=a^{\parallel d}u$. 
\end{thm}
\begin{proof}In first place note that if $d\in\hat{R}$, then $ud\in\hat{\R}$. In fact, if $\overline{d}\in\R$
is such that $d=d\overline{d}d$, then $ud= ud (\overline{d}u^{-1}) ud$. In addition, since $\overline{d}u^{-1} ud= \overline{d}d$, $(\overline{d}u^{-1}ud\R \overline{d}u^{-1}ud)^{-1}=(\overline{d}d\R \overline{d}d)^{-1}$
and $\overline{d}u^{-1} udau^{-1}ud=\overline{d}dad$. Consequently, according to Theorem \ref{thm8}, statements (i) and (ii) are 
equivalent. Furthermore, applying Theorem \ref{thm11} (ii), a direct calculation proves that $\R^{ud}=\R^{\parallel d}u^{-1}$.\par

\indent Let $a\in\R^{\parallel d}$. Then, according to Theorem \ref{thm11} (ii), there exist $z\in (\overline{d}d\R \overline{d}d)^{-1}$
and $m\in \overline{d}d\R(1- d\overline{d})+(1-\overline{d}d)\R$ such that $a=z\overline{d}+m$ and $a^{\parallel d}=(z\overline{d})^{\parallel d}=dw$,
where $w\in \overline{d}d\R \overline{d}d$ is such that $zw=wz=\overline{d}d$. Thus, $au^{-1}=z\overline{d}u^{-1}+mu^{-1}$ and 
according to Theorem \ref{thm11} (ii), $(au^{-1})^{\parallel ud}= (z\overline{d}u^{-1})^{\parallel ud}= udw=ua^{\parallel d}$.\par
\indent The equivalence between statements (i) and (iii) and the remaining identities can be proved using similar arguments.
\end{proof}

\section{Commutative inverses along an element}

\indent In this section, given a unitary ring $\R$, $d\in\hat{\R}$ and $a\in \R^{\parallel d}$, it will be characterized
when $aa^{\parallel d}= a^{\parallel d}a$. Note that in this case, since $a^{\parallel d}aa^{\parallel d}=a^{\parallel d}$,
$a^{\parallel d}\in \R^{\sharp}$, moreover, $(a^{\parallel d})^{\sharp}= a^2a^{\parallel d}$ (see Remark \ref{rema13} (iii)).
Next follows the first characterization of this section.\par

\begin{thm}\label{thm17}Let $\R$ be a unitary ring and consider $d\in\hat{\R}$ and $a\in\R^{\parallel d}$. Then,
the following statements are equivalents.\par
\noindent \rm (i) \it $aa^{\parallel d}= a^{\parallel d}a$.\par
\noindent \rm (ii) \it  $d\in \R^\sharp$ and $ap_d=p_da$.\par
\noindent \rm (iii) \it   $d\in \R^\sharp$ and $a=x+m$, where $x\in ((1-p_d)\R (1-p_d))^{-1}$ and $m\in p_d\R p_d$.
\end{thm}
\begin{proof}Note that according to Corollary \ref{cor21},
statement (ii) and (iii) are equivalent.\par

\indent Let $\overline{d}\in \R$
such that $d=d\overline{d}d$. Then, applying Theoem \ref{thm11} (i),
it is not difficult to prove that
\begin{align*}
aa^{\parallel d}&=\overline{d}d +(1-\overline{d}d)aa^{\parallel d},\\
a^{\parallel d}a&= d\overline{d} +a^{\parallel d}a (1-d\overline{d}).\\
\end{align*}
\indent Suppose that $aa^{\parallel d}= a^{\parallel d}a$. Then, 
$$
\overline{d}d=d\overline{d} +a^{\parallel d}a (1-d\overline{d})-(1-\overline{d}d)aa^{\parallel d}.
$$
In particular, multiplying by $d$ on the left:
$$
d=d^2\overline{d}+da^{\parallel d}a(1-d\overline{d}).
$$
\noindent Since $a^{\parallel d}\in d\R$, there is $x\in\R$ such that 
$$
d=d^2x.
$$
\noindent Similarly, since 
$$
d\overline{d}=\overline{d}d +(1-\overline{d}d)aa^{\parallel d}-a^{\parallel d}a (1-d\overline{d}),
$$
multiplying by $d$ on the right,
$$
d=\overline{d}d^2+(1-\overline{d}d)aa^{\parallel d}d.
$$
\noindent Since $a^{\parallel d}\in \R d$, there is $y\in\R$ such that 
$$
d=yd^2.
$$
Therefore, according to \cite[Theorem 4]{D}, $d\in \R^\sharp$.\par
\indent Next consider $d^\sharp$. Note that according to the structure 
of $a^{\parallel d}$ presented in Theorem \ref{thm11} (i) (using $\overline{d}=d^\sharp$),
it is easy to prove that
$$
a^{\parallel d}dd^\sharp=a^{\parallel d}=dd^\sharp a^{\parallel d},
$$
equivalently
$$
a^{\parallel d}p_d=0=p_d a^{\parallel d}. 
$$
Consequently, $p_d aa^{\parallel d}=0=a^{\parallel d}ap_d$ and
$$
aa^{\parallel d}=a^{\parallel d}a= dd^\sharp+a^{\parallel d}a(1-dd^\sharp)=dd^\sharp+a^{\parallel d}ap_d=dd^\sharp=d^\sharp d.
$$
As a result, 
$$
add^\sharp= aa^{\parallel d}a=dd^\sharp a,
$$
which implies that $ap_d=p_da$.\par
\indent To prove the converse, recall that according to Theorem \ref{thm14} (i),
$\R^{\parallel d}=\R^{\parallel 1-p_d}$. Moreover, since $ap_d=p_da$, according to Theorem \ref{thm11} (iii),
there exist $x\in (dd^\sharp\R dd^\sharp)^{-1}$ and $m\in p_d\R p_d$ such that
$a=x+m$. Moreover, according to Theorem \ref{thm11} (iii),
$a^{\parallel d}=w$, where $w\in dd^\sharp\R dd^\sharp$ is such that $xw=wx=dd^\sharp$.
However, a straightforward calculation proves that
$$
aa^{\parallel d}=xw=wx=a^{\parallel d}a=dd^\sharp.
$$ 
\end{proof}
\indent Thanks to Theorem \ref{thm17}, a characterization of group invertible elements
in terms of commuting inverse along an element can be derived.\par
\begin{cor}\label{cor19}Let $\R$ be a unitary ring and consider $d\in\hat{\R}$. Then,
the following statements are equivalents.\par
\noindent \rm (i) \it $d\in\R^\sharp$.\par
\noindent \rm (ii) \it There exists $a\in\R^{\parallel d}$ such that $aa^{\parallel d}= a^{\parallel d}a$.\par
\end{cor}

\begin{proof}If $d\in \R^\sharp$, then consider $a=d$. In fact, according to \cite[Theorem 11]{M},
$a^{\parallel d}$ exists, what is more, $a^{\parallel d}=d^\sharp$.\par

\indent On the other hand, if statements (ii) holds, then apply Theorem \ref{thm17}.
\end{proof}

\indent Next follows the second characterization of this section.\par

\begin{thm}\label{thm5000}Let $\R$ be a unitary ring and consider $a\in\R$ and $d \in \hat{\R}$ such that $a$ is invertible 
along $d$. Then, the following statements are equivalent.\par
\noindent \rm (i) \it $a^{\parallel d} a = a a^{\parallel d}$.\par 
\noindent \rm (ii) $da \in \R d$ and $ad \in d \R$.
\end{thm}
\begin{proof}Let $\overline{d}\in\R$ be such that $d=d \overline{d} d$.
Recall that according to  \cite[Theorem 3.2]{MP}, 
$da+1-d\overline{d}$ and $ad+1-\overline{d}d$ are invertible in $\R$ and
$a^{\parallel d} = (da+1-d\overline{d})^{-1}d = d(ad+1-\overline{d} d)^{-1}$.
Hence, 
\begin{align*}
\begin{split}
a^{\parallel d}a = a a^{\parallel d} & \Leftrightarrow 
(da+1-d\overline{d})^{-1}d a = a d(ad+1-\overline{d} d)^{-1} \\
& \Leftrightarrow d a (ad+1-\overline{d} d)= (da+1-d\overline{d}) a d \\
& \Leftrightarrow d a (1-\overline{d}d)= (1-d\overline{d}) a d.
\end{split}
\end{align*}

Assume that $a^{\parallel d}a = a a^{\parallel d}$. If the last equality is multiplied on the left by $d\overline{d}$, then $da(1-d^-d)=0$. 
Thus, $(1-dd^-)ad=0$. Therefore, $da = da\overline{d}d \in \R d$ and $ad=d\overline{d}ad \in d \R$.

To prove the converse, suppose that $da \in \R d$. Then there exists $u \in \R$ such that
$da = ud$. Consequently, $da\overline{d}d = ud\overline{d}d=ud=da$. 
Similarly, $ad \in d \R$ implies that $d\overline{d}ad=ad$.
As a result, $d a (1-\overline{d}d)= 0=(1-d\overline{d}) a d$, which is equivalent to $a^{\parallel d}a = a a^{\parallel d}$.

\end{proof}

\begin{rema}\label{rema6000}\rm Let $\R$ be a unitary ring with involution and consider
$a\in\R$ a Moore-Penrose invertible element. From Theorem \ref{thm5000} follows that
necessary and sufficient for $a$ to be EP is that 
$aa^* \in a^*\R$ and $a^*a \in \R a^*$.
\end{rema}

\section{Inner inverse}

\indent In this section inverses along an element that are also inner inverses will be studied.
Next follows a characterization of this object.\par

\begin{thm}Let $\R$ be a unitary ring and consider $a\in\R$ and $d \in \hat{\R}$ such that $a$ is invertible along $d$. The
following statements are equivalent:\par
\noindent {\rm (i)} $a^{\parallel d}$ is an inner inverse of $a$.\par
\noindent {\rm (ii)} $\R=d \R \oplus a^{-1}(0)$.\par
\noindent {\rm (iii)} $\R=\R d \oplus a_{-1}(0)$.	
\end{thm}
\begin{proof}Since $a$ is invertible along $d$, $a^{\parallel d}\R= d\R$ and 
$\R a^{\parallel d}=\R d$. Apply then Remark \ref{rema1} (iv). 
\end{proof}

\indent In the following theorem some special cases will be considered.\par

\begin{thm}\label{thm701}Let $\R$ be a unitary ring and consider $a\in\R$ and $d \in \hat{\R}$ such that $a$ is invertible along $d$
and  $a^{\parallel d}$ is an inner inverse of $a$. The following statements hold.\par
\noindent {\rm (i)} If $\overline{d}$ is an inner inverse of $d$, then $a^{\parallel d}$ is an inner and outer inverse of $\overline{d}dad\overline{d}$.\par
\noindent {\rm (ii)} If $d$ is group invertible, then $a^{\parallel d} = (d^\# d a dd^\#)^\#$.\par
\noindent {\rm (iii)} If the ring $\R$ has an involution and $d$ is Moore-Penrose invertible, then $a^{\parallel d} = (d^\dag d a dd^\dag)^\dag$.
\end{thm}
\begin{proof}Note that according to Definition \ref{def2}, $a^{\parallel d}\overline{d}d=a^{\parallel d}=d\overline{d}a^{\parallel d}$.\par 

\noindent (i). From the previous observation, 
$$
a^{\parallel d}(\overline{d}dad\overline{d})a^{\parallel d} = (a^{\parallel d}\overline{d}d)a(d\overline{d}
a^{\parallel d}) = a^{\parallel d}aa^{\parallel d}= a^{\parallel d}
$$

 and 
$$(\overline{d}dad\overline{d})a^{\parallel d}(\overline{d}dad\overline{d}) = \overline{d}da(d\overline{d}a^{\parallel d}\overline{d}d)ad\overline{d} =
 \overline{d}daa^{\parallel d}ad\overline{d} = \overline{d}dad\overline{d}.
$$
	
\noindent (ii). It remains to prove that $a^{\parallel d}(d^\sharp d a d d^\sharp) = (d^\sharp d a d d^\sharp ) a^{\parallel d}$. 
This follows from $a^{\parallel d}(d^\sharp d a d d^\sharp) = (a^{\parallel d}d^\sharp d)add^\sharp = a^{\parallel d}add^\sharp=dd^\sharp$ and 
	$(d^\sharp d a d d^\sharp )a^{\parallel d} = d^\sharp d a a^{\parallel d}=d^\sharp d$ (Theorem \ref{thm11} (i)-(ii)).
	
\noindent (iii). It is enough to prove that $a^{\parallel d}(d^\dag d a d d^\dag)$ and 
	$(d^\dag d a d d^\dag)a^{\parallel d}$ are Hermitian.
	In fact,  $a^{\parallel d}(d^\dag d a d d^\dag) = dd^\dag$ and 
	$(d^\dag d a d d^\dag)a^{\parallel d} = d^\dag d$.
\end{proof}

\indent In particular, when EP elements are considered, new expressions of the group and Moore-Penrose
inverse can be obtained.

\begin{cor}\label{cor705}Let $\R$ be a unitary ring with an involution and let $a \in \R$
be EP. The following statements hold.\par 
\noindent {\rm (i)} $a^\dag = ((aa^\sharp )^* a (a^\sharp a)^*)^\sharp$.\par
\noindent {\rm (ii)} $a^\sharp = (a^\dag a^3a^\dag)^\dag$.
\end{cor}
\begin{proof}Recall that $a^\sharp=a^\dag$, $a^\dag=a^{\parallel a^*}$ and $a^\sharp=a^{\parallel a}$
(\cite[Theorem 11]{M}). In addition, recall that $a^*$ is group invertible and $(a^*)^\sharp=(a^\sharp)^*$.
Apply then Theorem \ref{thm701}.
\end{proof}

\indent To present more expressions of the group and the Moore-Penrose inverse,
the following theorem will be useful. 

\begin{thm}\label{thm702}Let $\R$ be a unitary ring and consider $a \in \R$ and $d\in\hat{\R}$ such that $a$ invertible along $d$. 
Then, if $x$ is an inner inverse of $dad$, $a^{\parallel d} = dxd$.
\end{thm}
\begin{proof}Since $dadxdad=dad$, $d(adxdad-ad)=0$. According to Theorem \ref{thm5} (xiii), $a^{\parallel a}(adxdad-ad)=0$, i.e.,
	$a^{\parallel a}adxdad=a^{\parallel a}ad$. According again to Theorem \ref{thm5} (xiii), $a^{\parallel a}adxdaa^{\parallel a}=a^{\parallel a}aa^{\parallel a}=a^{\parallel a}$. However, according to \cite[Lemma 3]{M}, $a^{\parallel a}ad =d=daa^{\parallel a}$. Therefore,
$dxd=a^{\parallel a}$.
\end{proof}

\begin{cor}Let $\R$ be a unitary ring and consider $a\in \R$. The following statements hold.\par
\noindent {\rm (i)} If $a$ is group invertible and $\overline{a}$ is an inner inverse of $a$, then $a\overline{a}a^\sharp = a^\sharp = a^\sharp \overline{a}a$.\par
\noindent {\rm (ii)} If $a$ is group invertible and $x$ is an inner inverse of $a^3$, then $a^\sharp = axa$.\par 
\noindent If in addition $\R$ has an involution, then the following statements hold.\par
\noindent {\rm (iii)} If $a$ is Moore-Penrose invertible and $\overline{a}$ is an inner inverse of $a$, then $a^\dag (a\overline{a})^* = a^\dag = 
(\overline{a}a)^* a^\dag$.\par
\noindent {\rm (iv)} If $a$ is Moore-Penrose invertible and $x$ is an inner inverse of $a^*aa^*$, then $a^\dag = a^*xa^*$.
\end{cor} 
\begin{proof}To prove statement (i) (respectively statement (iii)) recall that according to \cite[Theorem 11]{M} $a^\sharp=a^{\parallel a}$ (respectively $a^\dag=a^{\parallel a^*}$). Then apply $a^{\parallel d}\overline{d}d=a^{\parallel d}=d\overline{d}a^{\parallel d}$ (see the proof of Theorem \ref{thm701}) to
$d=a$ and $\overline{d}=\overline{a}$ (respectively  $d=a^*$ and $\overline{d}=(\overline{a})^*$).\par
\indent To prove statement (iii) (respectively statement (iv)), use that $a^\sharp=a^{\parallel a}$ (respectively $a^\dag=a^{\parallel a^*}$) (\cite[Theorem 11]{M}) and then apply  Theorem  \ref{thm702} with $d=a$ (respectively $d=a^*$).
\end{proof}

\bibliographystyle{amsplain}

\bigskip
\noindent Julio Ben\'{\i}tez\par
\noindent E-mail address: jbenitez@mat.upv.es \par
\medskip
\noindent Enrico Boasso\par
\noindent E-mail address: enrico\_odisseo@yahoo.it
\end{document}